\documentclass{lms}
\usepackage{amsmath, amssymb, algorithm, longtable}

\usepackage{tikz}
\usetikzlibrary{arrows}
\tikzstyle{block}=[draw opacity=0.7,line width=1.4cm]

\usepackage[all,arc,curve,frame,color]{xy}
\usepackage{subfigure}
\usepackage{url}

\newtheorem{thm}{Theorem}
\newtheorem{lem}{Lemma}
\newtheorem{prop}{Proposition}
\newtheorem{define}{Definition}

\newcommand{\ZZ}{\mathbb{Z}}     
\newcommand{\QQ}{\mathbb{Q}}      
\newcommand{\PP}{\mathbb{P}}      
\newcommand{\OO}{\mathcal{O}}    
\newcommand{\CC}{\mathbb{C}}      
\newcommand{\FF}{\mathbb{F}}      

\newcommand{\pp}{\mathfrak{p}}   

\newcommand{\var}{\tt}      

\newcommand{\Manoa}{M\=anoa}
\newcommand{\Hawaii}{Hawai\kern.05em`\kern.05em\relax i}

\providecommand{\PGL}{\mathop{\rm PGL}\nolimits}

\title[Quadratic PCF functions]{A census of quadratic post-critically finite\\ rational functions defined over $\QQ$}

 \date{\today}
\author[D. Lukas, M. Manes, D. Yap]{David Lukas, Michelle Manes, and Diane Yap}

\classno{37P05 (primary), 37P15, 11R99, 11Y99 (secondary)}

\extraline{The first author was partially supported by NSF grant DMS-1102858.}

\begin{document}

\maketitle

    \begin{abstract}
We find all quadratic post-critically finite (PCF) rational maps defined over $\QQ$.  We describe an algorithm to search for possibly PCF maps.  Using  the algorithm, we eliminate all but twelve rational maps, all of which are verifiably PCF.  We also give a complete description of possible rational preperiodic  structures  for quadratic PCF maps defined over $\QQ$.   
\end{abstract}

keywords: PCF, post-critically finite, quadratic rational maps, arithmetic dynamics

\section{Introduction}\label{sec:intro}

Let $\phi(z) \in \QQ(z)$ have degree $d\geq 2$.  We may regard $\phi:\PP^1 \to \PP^1$ as a morphism of the projective line.  We consider iterates of $\phi$:
\[
\phi^n(z) = \underbrace{\phi\circ\phi\circ\cdots\circ\phi}_{n\textup{ times}} (z),
 \quad \textup{ and } \quad \phi^0(z) = z.
\]
The orbit of a point $\alpha \in \PP^1$ is the set
$\OO_\phi(\alpha) = \left\{ \phi^n(\alpha) \mid n \geq 0 \right\}$.

Rather than studying individual rational maps, we  consider equivalence classes of  maps under conjugation by   $f \in\PGL_2(\CC)$; we define
$\phi^f = f \circ \phi \circ f^{-1}$. 
Note that $\phi$ and $ \phi^f$ have the same 
dynamical behavior.  In particular,  $f$ maps the $\phi$-orbit of $\alpha$ to the $\phi^f$-orbit of $f(\alpha)$.  

Critical points of $\phi$ are the points $\alpha \in \PP^1$ such that $\phi'(\alpha) = 0$ as long as $\alpha$ and $\phi(\alpha)$ are finite.  To compute the derivative at the excluded values of $\alpha$, we use a  conjugate map.  See~\cite[Section 1.2]{ads} for details.

\begin{define}
A rational map  $\phi: \PP^1 \to \PP^1$ of degree $d\geq 2$ is \emph{postcritically finite}  (PCF) if the  orbit of each  critical point is finite.  
\end{define}

A fundamental observation in the study of one-dimensional complex dynamics is that the orbits of the finite set of critical points of $\phi$  largely determines the dynamics of $\phi$ on all of $\PP^1$.  So the study of PCF maps has a long history in complex dynamics, including Thurston's topological characterization of these maps in the early 1980s and continuing to the present day.

  In~\cite{pilgrimcensus}, for example, the authors 
find exactly one representative from each conjugacy class of nonpolynomial hyperbolic
PCF  rational maps of degree 2 and 3 in which the post-critical set --- the forward orbit of the critical points, excluding the points themselves --- contains no more than four points.
In~\cite{MdADS}, Silverman advances the idea of PCF maps as a  dynamical analog of abelian varieties with complex multiplication, suggesting that these maps may be of special interest in arithmetic dynamics as well. 

Our main result is inspired by these ideas.  Compare this to the statement that, up to isomorphism over $\bar \QQ$, there are exactly thirteen elliptic curves $E / \QQ$ with complex multiplication.

\begin{thm}\label{all pcf}
There are exactly twelve $\bar \QQ$ conjugacy classes of quadratic PCF maps defined over $\QQ$\textup{:}

\begin{tabular} {l    l      l      l}
\textup{(1)} \  $z^2$    &
\textup{(2)} \  $\dfrac{1}{z^2}$ &
\textup{(3)} \   $z^2-2$  &
\textup{(4)} \  $z^2-1$\\
\textup{(5)} \  $\dfrac{1 }{ 2(z-1)^2}    $ & 
\textup{(6)} \ $ \dfrac{1 }{ (z - 1)^2} $& 
\textup{(7)} \ $ \dfrac{-1}{4z^2-4z}$& 
\textup{(8)} \ $\dfrac{-4}{9z^2-12z}$\\
\textup{(9)}\  $\dfrac{2}{(z-1)^2} $ &
\textup{(10)}\  $\dfrac{2 z+1}{4 z-2 z^2} $& 
\textup{(11)}\  $\dfrac{-2 z}{2 z^2-4 z+1} $& 
\textup{(12)}\ $ \dfrac{3 z^2-4 z+1}{1-4 z}$
\end{tabular}

\end{thm}

Of these, the first four were well-known to researchers in both complex and arithmetic dynamics.  Maps $(5)$--$(8)$ appeared in~\cite{pilgrimcensus}.     Maps (9)--(12) did not appear in~\cite{pilgrimcensus} because they fail to fit either the criterion of hyperbolicity or the post critical set is too large.  One major contribution of this work is the fact that this list is exhaustive.  

A fundamental problem in arithmetic dynamics is classifying rational functions by the structure of their rational preperiodic points.  In~\cite{poonenrefined}, Poonen undertakes this task for quadratic polynomials defined over $\QQ$, subject to the condition that no rational point is on a cycle of length greater than~$3$.  In~\cite{manespreper}, Manes gives a  classification for quadratic rational maps with nontrivial $\PGL_2$ stabilizer, subject to a similar condition.  
Given the comprehensive list  in Theorem~\ref{all pcf}, we  are able to describe all possible  rational preperiodic structures for quadratic PCF maps defined over $\QQ$ with no additional hypotheses.  Difficulty arises only for the first two maps, which have nontrivial twists.  We are able to conclude the following.

\begin{thm}\label{cor:preper}
A quadratic PCF map defined over $\QQ$ has at most six rational preperiodic points.  
\end{thm}

Given the parallels between the set of rational preperiodic points for a rational map and the torsion subgroup of an abelian variety $A(\QQ)$ (see~\cite[page 111]{MdADS}, for example) this result and the preperiodic structures given in Sections~\ref{sec:PCFnoSym} and~\ref{sec:symmlocus} are analogs of the comprehensive list of torsion subgroups for CM elliptic curves $E/\QQ$  in~\cite{OlsonCM}.

\section{Background}\label{sec:bg}

Since $\phi(z) \in \QQ(z)$ is quadratic, the fixed points of $\phi$ are roots of a cubic polynomial found by setting $\phi(z)=z$.  The multiplier at a finite fixed point $\alpha$ is  $\phi'(\alpha)$, and when the point at infinity is fixed we compute the multiplier for that point using a conjugate map.   The following result is used to iterate through equivalence classes of quadratic rational maps with trivial $\PGL_2$ stabilizer. Determining which rational functions  with nontrivial stabilizer are PCF is  addressed  in Section~\ref{sec:symmlocus}.

\begin{thm}\cite[Lemma 3.1]{MMYY}
\label{MMYY normal form} Let $K$ be a field with characteristic different from $2$ and~$3$.
Let $\psi(z) \in K(z)$ have degree~$2$, and let $\lambda_1, \lambda_2, \lambda_3 \in \overline K$ be the multipliers of the fixed points of $\psi$ \textup(counted with multiplicity\textup). Then $\psi(z)$ is conjugate over $K$ to the map
\begin{equation}
\phi(z) = \frac{2 z^2 + (2-\sigma_1) z + (2-\sigma_1)}{-z^2 + (2+\sigma_1)z + 2-\sigma_1-\sigma_2} \in K(z),
\label{eqn: normal form}
\end{equation}
where $\sigma_1 $ and $\sigma_2$ are the first two symmetric functions of the multipliers.  Furthermore, no two distinct maps of this form are conjugate to each other over~$\overline K$.

\end{thm}

 The following result gives a crucial bound on $H(\lambda)$, the standard multiplicative height of a fixed point multiplier for a quadratic PCF map.  (See~\cite[Section 3.1]{ads} for background on heights.)    Using the lemma, we  derive explicit height bounds for $\sigma_1$ and $\sigma_2$.

\begin{lem} \cite[Corollary 1.3]{PCF_height}
\label{Height Bound} 
Let $\phi(z) \in \overline{\QQ}(z)$ have degree~$2$, suppose that $\phi$ is PCF, and let $\lambda$ be the multiplier of any fixed point of $\phi$. Then $H(\lambda) \leq  4$.\end{lem}

\begin{prop}\label{derived height bound}
 Let $\phi(z)\in \overline{\QQ}$ be a degree $2$ PCF map, and suppose that $\sigma_1$ and $\sigma_2$ are the first and second symmetric functions on the multipliers of the fixed points. Then $H(\sigma_1)\leq 192$ and $H(\sigma_2)\leq 12288$.
 \end{prop}

\begin{proof*} 
We simplify notation by setting $d=[K:\QQ]$ for $K$ any field of definition of the fixed point multipliers.
By the triangle inequality: 
\[ 
|\sigma_1|_v = |\lambda_1 +\lambda_2+ \lambda_3|_v \leq
\begin{cases}
 \max{\{|\lambda_1|_v, |\lambda_2|_v, |\lambda_3|_v}\}
 & \text{for each finite place}\\
 3\max{\{|\lambda_1|_v, |\lambda_2|_v, |\lambda_3|_v}\}
& \text{for each infinite place.}
\end{cases}
\]

 For an extension of degree $d$, there are at most $d$ infinite places, so 
\begin{eqnarray*}
H(\sigma_1) &= & 
\prod_{v\in M_K} \left(\max\{|\sigma_1|_v, 1\}^{n_v}\right)^{1/d}
\leq 3\prod_{v\in M_K} \left(\max_{1\leq i \leq 3}\{|\lambda_i|_v, 1\}^{n_v}\right)^{1/d}\\
&\leq&
 3\prod_{v\in M_K}\left(\max\{|\lambda_1|_v, 1\}^{n_v}\cdot \max\{|\lambda_2|_v, 1\}^{n_v}\cdot \max\{|\lambda_3|_v, 1\}^{n_v}\right)^{1/d}\\
&= &
3H(\lambda_1)H(\lambda_2)H(\lambda_3) \leq 3\cdot4^3 = 192.
\end{eqnarray*}

The proof for the bound on $\sigma_2$ follows similarly:
\multbox
\begin{eqnarray*}
H(\sigma_2) 
&=&
 H(\lambda_1\lambda_2+\lambda_1\lambda_3+\lambda_2\lambda_3)
\leq 3\prod_{v\in M_K} \left(\max_{i\neq j \atop 1\leq i, j \leq 3}\{|\lambda_i\lambda_j|_v, 1\}^{n_v}\right)^{1/d}\\
& \leq & 
3\prod_{v\in M_K}\left(\max\{|\lambda_1\lambda_2|_v, 1\}^{n_v}\cdot \max\{|\lambda_2\lambda_3|_v, 1\}^{n_v}\cdot \max\{|\lambda_1\lambda_3|_v, 1\}^{n_v}\right)^{1/d}\\
&= &
3H(\lambda_1\lambda_2)H(\lambda_2\lambda_3)H(\lambda_1\lambda_3)
 \leq 3\cdot 4^6 = 12288. 
\end{eqnarray*}
\emultbox
\end{proof*}

This height bound coupled with the normal form of Theorem \ref{MMYY normal form} allows us to:
(1) iterate through possible rational values of $\sigma_1$ and $\sigma_2$,
(2) form a unique rational map from each equivalence class, and 
 (3) test if that map is PCF.

Step (3)  relies on the following two theorems.  For notation: $K$ is a local field with nonarchimedean absolute value $|\cdot|_v$,  $R$ is the ring of integers of $K$,  $\pp$ is the maximal ideal of $R$, $k = R/\pp$ is the residue field, and $\widetilde{\cdot}$ represents reduction  modulo $\pp$.  A morphism $\phi$ has good reduction at  $\pp$ if $\deg(\phi) = \deg(\widetilde{\phi})$.

\begin{thm}\cite[Theorem 2.21]{ads}
\label{periods}  Let $\phi: \PP^1 \to \PP^1$ be a rational function of degree $d\geq 2$ defined over $K$. Assume that $\phi$ has good reduction,  let $P\in \PP^1(K)$ be a periodic point of $\phi$, and define the following quantities:

\begin{tabular}{ll}
$n$& The exact period of $P$ for the map $\phi$.\\
$m$& The exact period of $\widetilde{P}$ for the map $\widetilde{\phi}$.\\
$r$& The order of $\lambda_{\widetilde{\phi}}(\widetilde{P}) = (\widetilde{\phi}^m)'(\widetilde{P})$ in $k^{\times}$. 
\textup(Set $r=\infty$ if $\lambda_{\widetilde{\phi}}(\widetilde{P})$ is not a root of unity.\textup)\\
$p$ & The characteristic of the residue field $k$.\\
\end{tabular}

Then $n$ has one of the following forms:
\[n = m \hspace{5mm}\text{or}  \hspace{5mm}n = mr  \hspace{5mm}\text{or}  \hspace{5mm}n=mrp^e.\]
\end{thm}

Let $\phi(z) \in \QQ(z)$ have critical points $\gamma_1, \gamma_2$. To apply Theroem~\ref{periods},   we consider $\phi$ to be defined over $\QQ_p$ for $p$ a prime of good reduction and with $\gamma_i \in \QQ_p$. 
 If $\phi$ is PCF, then some iterate $\phi^j(\gamma_i)$ has exact period $n$.  
Since the $\FF_p$-orbit $\OO_{\tilde\phi}(\tilde \gamma_i)$ is necessarily finite,  some iterate $\tilde\phi^k(\tilde \gamma_i)$  has exact period $m$.  Theorem~\ref{periods} gives  a set of possible   $n$ values based on the more-easily computed~$m$.

Because we focus on rational functions defined over $\QQ$, we can use~\cite[Theorem 2.28]{ads} to refine the statement above.  We conclude that $e=0$ when $p \neq 2$, and $e\in\{0,1\}$ when $p=2$. In the algorithm, we exclude $p=2$ from consideration, so we have  
$n = m$ or   $n = mr$.

\section{Algorithm}\label{sec:algorithm}
We first describe the overall flow of the algorithm for determining if a map with trivial stabilizer is potentially PCF.  We then provide more detailed pseudocode and explanation for each piece of the algorithm.  Our actual code is provided in the arXiv distribution of this article.
\begin{enumerate}
\item
Build  a database containing, for each $p$ in a list of primes, all quadratic rational maps mod~$p$ in the form given in~\eqref{eqn: normal form}, along with their critical points and corresponding possible global periods as given by Theorem~\ref{periods}.

\item
For by  $\sigma_1$ and $\sigma_2$ values within the height bounds in Proposition~\ref{derived height bound}, create $\phi(z) \in \QQ(z)$ as in~\eqref{eqn: normal form},  reduce the map modulo primes $p$, and iteratively intersect the possible global periods (found in the database) for the critical points at each prime.  If the intersection becomes empty, the map is not PCF.

\item
For maps returned by the algorithm as possibly PCF, calculate the forward orbits of both critical points to verify that they are, indeed, finite.

\end{enumerate}

Algorithm~\ref{build_db} builds the databases; our  implementation used the first 130 odd primes.    Note that a quadratic rational map over $\FF_p$ for $p$ an odd prime has exactly two critical points~\cite[Corollary 1.3]{unicrit}, which are the roots of the wronskian (together with the point at infinity in the case the wronskian is linear).  If the wronskian is an irreducible quadratic over $\FF_p$, the map is not included in the database.

\begin{algorithm}[h] 
\begin{flushleft}
\smallskip

Input: {\var pList}, a list of odd primes  

\smallskip

Output: a database of quadratic rational maps over $\FF_p$ together with critical point data for each prime $p$ in {\var pList}
\bigskip

for $p$ in {\var pList}:\\

\medskip

\quad for each pair $(b,c) \in \FF_p^2$:\\
\quad\quad create the morphism $f := [2x^2 + bxy + by^2, -x^2 + (4-b)xy + cy^2]$\\

\medskip

\quad\quad if  $\deg(f)=2$  and $c_1, c_2 :=$  the critical points of $f$ are defined over $\FF_p$:\\
\quad\quad\quad add an entry for $f$ to the database\\

\medskip

\quad\quad\quad for $i = 1, 2$:\\
\quad\quad\quad\quad find  $m_i$, the length of the cycle into which $c_i$'s orbit eventually falls\\
\quad\quad\quad\quad find $\lambda_i$,  the multiplier of that cycle\\

\medskip

\quad\quad\quad\quad if $\lambda_i = 0$:\\
\quad\quad\quad\quad\quad append the pair $(c_i, \{m_i\})$  to the database entry for $f$ \\
\quad\quad\quad\quad else:\\
\quad\quad\quad\quad\quad find $r_i := $ the multiplicative order of $\lambda_i$ in $\FF_p^{\times}$\\
\quad\quad\quad\quad\quad append the pair $(c_i, \{m_i, m_ir_i\})$ to the database entry for $f$ \\

\end{flushleft}
\caption{--- \textbf{Build Database}}
\label{build_db}
\end{algorithm}

Algorithm~\ref{outer}  filters out functions which are certainly not PCF, but does not guarantee that the functions which remain are PCF. The algorithm identified ten potentially PCF quadratic rational functions, all of which were verified to be PCF by  iterating the function at each critical point.  These maps are described in Section~\ref{sec:PCFnoSym}.  The algorithm uses the resultant of the rational map $\phi$,  meaning the resultant of the relatively prime polynomials $f$ and $g$ such that $\phi =f/g$.  The resultant of a map given in the form~\eqref{eqn: normal form} is nonzero if and only if the map has trivial stabilizer~\cite[Remark~3.2]{MMYY}, and primes dividing the resultant are precisely the primes of bad reduction for $\phi$~\cite[Sections~2.4 and 2.5]{ads}.

\begin{algorithm}[h] 
\begin{flushleft}
\smallskip
Input:  {\var pList}, a list of primes included in the database;  height bounds $H_1$ and $H_2$\\

\smallskip
Output: a set of possibly PCF maps with $\sigma_1 \leq H_1$ and $\sigma_2 \leq H_2$\\

\bigskip

for $\sigma_1 \in \QQ$ of height $\leq 192$ and $\sigma_2\in \QQ$ of height $\leq 12288$:  \\
	\quad create the rational map 
	$\phi(z) := \frac{2z^2 + (2-\sigma_1)z + (2-\sigma_1)}{-z^2 + (2+\sigma_1)z + (2-\sigma_1-\sigma_2)}$\\	
	 \quad normalize the coordinates of $\phi$ (clearing denominators so  all coefficients are in $\ZZ$)\\
	 \quad calculate ${\var res}:= $ resultant of $\phi$\\
	 
	 \medskip
	 
	 \quad if ${\var res} \neq 0$:\\
	 \quad\quad calculate $\gamma_1, \gamma_2 :=$ critical points of $\phi$\\
	
	\medskip

	\quad \quad if $\gamma_1, \gamma_2 \in \QQ$:\\
	\quad \quad\quad if {\var Check\_Rational\_Periods}($\phi$, $\gamma_1$, $\gamma_2$, 
	{\var pList}, {\var res}):\\
	\quad \quad\quad\quad add $\phi$ to set of possibly PCF maps\\
	\quad \quad else:\\
	\quad \quad\quad if {\var Check\_Irrational\_Periods}($\phi$,  
	{\var pList}, {\var res}):\\
	\quad \quad\quad\quad add $\phi$ to set of possibly PCF maps\\

\end{flushleft}
\caption{--- Find PCF maps up to height bound}
\label{outer}
\end{algorithm}

 Algorithm~\ref{alg:CheckRat} is called when $\phi$ has rational   critical points $\gamma_1$ and $\gamma_2$, meaning we can easily reduce them modulo primes $p$.  (If $p$ divides the denominator, the point reduces to the point at infinity on $\PP^1_{\FF_p}$.)  We must keep track of the possible periods for each critical point independently, since it is possible that they terminate in cycles of different lengths.  
 
 If  $\gamma_1$ and $\gamma_2$  are irrational, they must be Galois conjugates since the equation defining them is a quadratic polynomial in $\QQ[z]$.  Since $\phi(z) \in \QQ(z)$,  the same is true of $\phi^i(\gamma_1)$ and $\phi^i(\gamma_2)$ for every $i \geq 0$.  Therefore, the orbits of $\gamma_1$ and $\gamma_2$ are either both finite or both infinite, and if they are finite they will terminate in cycles of the same length.  Algorithm~\ref{alg:checkIrr} takes advantage of this symmetry.  Note that we do not  reduce irrational critical points modulo primes; we simply check if the reduced function appears in the database or not.

\begin{algorithm}[h] 
\begin{flushleft}
\smallskip

Input: a quadratic rational map $\phi$ with integer coefficients, resultant {\var res}, and rational critical points $\gamma_1, \gamma_2$.  A list of primes {\var pList} for which the database has been built.\\

\smallskip

Output: {\var False} if $\phi$ is definitely not PCF and {\var True} otherwise\\

\bigskip

initialize empty lists ${\var PossPer}_1$ and ${\var PossPer}_2$\\

\medskip

for $p$ in {\var pList}:\\
	\quad if $p \nmid {\var res}$ ($p$ is a prime of good reduction):\\
	\quad \quad set $f:= \phi \pmod p$,  \quad $c_1:= \gamma_1 \pmod p$, \quad $c_2:= \gamma_2 \pmod p$\\
	\quad \quad look up $f$ in the database\\
	
	\medskip
	
	\quad \quad for $i = 1,2$:\\
	\quad \quad\quad retrieve the  set of possible global periods for $c_i$  \\
	\quad \quad\quad if ${\var PossPer}_i$ is empty (this is the first good prime):  \\
	\quad \quad\quad\quad set ${\var PossPer}_i =  \{ \text{possible global periods for } c_i \}$ \\
	\quad \quad\quad else:  \\
	\quad \quad\quad\quad set ${\var PossPer}_i = {\var PossPer}_i \cap \{ \text{possible global periods for } c_i\}$ \\	
	
	\medskip
	
	\quad \quad\quad if ${\var PossPer}_i$ is empty: \\
	\quad \quad\quad\quad return {\var False}\\
	
	\medskip

return {\var True}\\

\end{flushleft}
\caption{--- \textbf{Check\_Rational\_Periods} filters out maps $\phi$ which are not PCF}
\label{alg:CheckRat}
\end{algorithm}

\begin{algorithm}[h]
\begin{flushleft}
\smallskip

Input: a quadratic rational map $\phi$ which has irrational critical points and  integer coefficients.  The resultant of $\phi$ {\var res}.  A list of primes {\var pList} for which the database has been built.

\smallskip

Output: {\var False} if $\phi$ is definitely not PCF and {\var True} otherwise\\

\bigskip

initialize empty list {\var PossPer} \\

\medskip

for $p$ in {\var pList}:\\
	\quad if $p \nmid {\var res}$ ($p$ is a prime of good reduction):\\
	\quad \quad set $f:= \phi \pmod p$\\
	\quad \quad look up $f$ in the database\\
	\quad \quad let ${\var Poss}_1$ and ${\var Poss}_2$ be the sets of 
	possible global periods for the critical points of $f$\\

	\medskip
	
	\quad \quad if  {\var PossPer} is empty (this is the first good prime):  \\
	\quad \quad\quad set ${\var PossPer} =  {\var Poss}_1 \cap {\var Poss}_2$ \\
	\quad \quad else:  \\
	\quad \quad\quad set ${\var PossPer} = {\var PossPer} \cap {\var Poss}_1\cap {\var Poss}_2$ \\
	
	\medskip
	
	\quad \quad if {\var PossPer} is empty:\\
	\quad \quad\quad return {\var False}\\
	
	\medskip

return {\var True}\\

\end{flushleft}
\caption{--- \textbf{Check\_Irrational\_Periods} filters out maps $\phi$ which are not PCF}
\label{alg:checkIrr}
\end{algorithm}

 Algorithm~\ref{build_db} was implemented in Sage~\cite{sage}, using built-in functionality for morphisms on projective spaces over finite fields.   The database used throughout was GNU dbm~\cite{gdbm}.   Algorithms~\ref{outer}, \ref{alg:CheckRat}, and~\ref{alg:checkIrr} were prototyped in Sage and eventually implemented in C to improve speed of computation.  They  use the GNU Multiple Precision Arithmetic Library~\cite{GMP}.  The program was run on two 6-core Intel\textregistered\  Xeon\textregistered\ CPUs at 2.80GHz, with 12 GB of RAM   and running Linux (CentOS 5.10).

\section{PCF maps with trivial $\PGL_2$ stabilizer}\label{sec:PCFnoSym}
Table \ref{tab: trivial stab} lists the output of our algorithm:  all quadratic PCF maps defined over $\QQ$ with trivial $\PGL_2$ stabilizer.    In the critical portraits, an arrow from $P$ to $Q$ indicates that $\phi(P)=Q$; an integer over the arrow indicates the ramification index of the map at that point.  In particular, the critical points are the initial points for arrows where the integer is~$2$.    The final column gives a conjugate map in simpler form, reflecting the statement in Theorem~\ref{all pcf}.

\begin{center}
\begin{table}
\caption{All quadratic PCF maps defined over $\QQ$ with trivial $\PGL_2$ stabilizer}
\begin{tabular}{|c|c|c|}\hline

$\phi(z)$						& Critical portrait	& Conjugate map  \\ 
\hline \hline 

$\displaystyle\frac{2z^2}{-z^2+4z+8}$			& \xygraph{ 
		!{<0cm,0cm>;<1cm,0cm>:<0cm,1cm>::} 
		!{(-2,0) }*+{\bullet_{0}}="a1" 
		!{(0,0) }*+{\bullet_{-4}}="a" 
		!{(1.75,0) }*+{\bullet_{-4/3}}="b" 
		!{(3.5,0) }*+{\bullet_{4}}="c" 
		"a1" :@(r,lu)_2 "a1"
		"a":"b"^2 
		"b":"c" ^1
		"c" :@(r,lu)_1 "c"
	} & $z^2-2$\\ 
\hline

$\displaystyle\frac{2z^2}{-z^2+4z+4}$		 & \xygraph{ 
		!{<0cm,0cm>;<2cm,0cm>:<0cm,1cm>::} 
		!{(0,0) }*+{\bullet_{0}}="a" 
		!{(1,0) }*+{\bullet_{-2}}="b" 
		!{(2,0) }*+{\bullet_{-1}}="c" 
		"a":@(r,lu)_2 "a"
       "b":@/_/_2 "c"
       "c":@/_/_1 "b"
	} & $z^2-1$\\ 
 \hline

$\displaystyle\frac{2z^2+8z+8}{-z^2-4z+4}$	 	&\xygraph{ 
		!{<0cm,0cm>;<1.5cm,0cm>:<0cm,1cm>::} 
		!{(0,0) }*+{\bullet_{\infty}}="a" 
		!{(1,0) }*+{\bullet_{-2}}="b" 
		!{(2,0) }*+{\bullet_{0}}="c" 
		!{(3,0) }*+{\bullet_{2}}="d" 
		!{(4,0) }*+{\bullet_{-4}}="e" 
		"a":^2"b" 
		"b":^2"c" 
		"c":^1"d"
		"d":@/_/_1"e" 
		"e":@/_/_1"d" 
	}& $\displaystyle\frac{1}{2(z-1)^2}$ \\ 
\hline

$\displaystyle\frac{2z^2+8z+8}{-z^2-4z}$	 	&  \xygraph{ 
		!{<0cm,0cm>;<2cm,0cm>:<0cm,1cm>::} 
		!{(-0.6,-.25) }*+{\bullet_{-2}}="a" 
		!{(0.6,-.25) }*+{\bullet_{0}}="b" 
		!{(0,.75) }*+{\bullet_{\infty}}="c" 
		"a":@/_/_2"b"
    		"b":@/_/_1"c"
      		"c":@/_/_2"a"
	} &  $\displaystyle\frac{1}{(z-1)^2}$\\ 
 \hline

$\displaystyle\frac{2z^2+4z+4}{-z^2}$	 		& \xygraph{ 
		!{<0cm,0cm>;<2cm,0cm>:<0cm,1cm>::} 
		!{(0,0) }*+{\bullet_{0}}="a" 
		!{(1,0) }*+{\bullet_{\infty}}="b" 
		!{(2,0) }*+{\bullet_{-2}}="c" 
		!{(3,0) }*+{\bullet_{-1}}="d" 
		"a":^2"b" 
		"b":^1"c" 
		"c":@/_/_2"d" 
		"d":@/_/_1"c" 
	} &  $\displaystyle\frac{-1}{4z^2-4z}$ \\ 
\hline

$\displaystyle\frac{6z^2 + 8z + 8 }{ -3z^2 + 4z + 4}$			& \xygraph{ 
		!{<0cm,0cm>;<1.5cm,0cm>:<0cm,1cm>::} 
		!{(0,0) }*+{\bullet_{0}}="a" 
		!{(1,0) }*+{\bullet_{2}}="b" 
		!{(2,0) }*+{\bullet_{\infty}}="c" 
		!{(3,0) }*+{\bullet_{-2}}="d" 
		!{(4,0) }*+{\bullet_{-1}}="e" 
		"a":^2"b" 
		"b":^1"c" 
		"c":^1"d"
		"d":@/_/_2"e" 
		"e":@/_/_1"d" 
	}& $\displaystyle\frac{-4}{9z^2-12z}$\\  
\hline

$\displaystyle\frac{2z^2 + 8z + 8}{ -z^2 - 4z - 2}$			& \xygraph{ 
		!{<0cm,0cm>;<1.5cm,0cm>:<0cm,1cm>::} 
		!{(0,0) }*+{\bullet_{\infty}}="a1" 
		!{(1,0) }*+{\bullet_{-2}}="a" 
		!{(2,0) }*+{\bullet_{0}}="b" 
		!{(3,0) }*+{\bullet_{-4}}="c" 
		"a1":^2"a"
		"a":^2"b" 
		"b":^1"c" 
		"c" :@(r,lu)_1 "c"
	}& $\displaystyle\frac{2}{(z-1)^2}$\\ 
\hline

$\displaystyle\frac{2z^2 + 4z + 4 }{ -z^2 +  4}$			& \xygraph{ 
		!{<0cm,0cm>;<1.5cm,0cm>:<0cm,1cm>::} 
		!{(-0.75,1) }*+{\bullet^{-3-\sqrt{5}}}="a1" 
		!{(-0.75,-1) }*+{\bullet_{-3+\sqrt{5}}}="b1" 
		!{(.75,0.5) }*+{\bullet_{-\frac{1}{2}\left(1+\sqrt{5}\right)}}="a2" 
		!{(.75,-0.5) }*+{\bullet_{\frac{1}{2}\left(-1+\sqrt{5}\right)}}="b2" 
		!{(2.25,0) }*+{\bullet_{2}}="c" 
		!{(3,0) }*+{\bullet_{\infty}}="d" 
		!{(4,0) }*+{\bullet_{-2}}="e" 
		"a1":^2"a2" 
		"b1":_2"b2" 
		"a2":^1"c"
		"b2":_1"c"
		"c":^1"d"
		"d":@/_/_1"e" 
		"e":@/_/_1"d" 
	}& $\displaystyle\frac{2 z+1}{4 z-2 z^2}$\\ 
\hline

$\displaystyle\frac{2z^2 + 4z + 4 }{ -z^2 +2}$			& \xygraph{ 
		!{<0cm,0cm>;<1.5cm,0cm>:<0cm,1cm>::} 
		!{(-0.75,1) }*+{\bullet_{-2-\sqrt{2}}}="a1" 
		!{(-0.75,-1) }*+{\bullet_{-2+\sqrt{2}}}="b1" 
		!{(.75,0.5) }*+{\bullet_{-\sqrt{2}}}="a2" 
		!{(.75,-0.5) }*+{\bullet_{\sqrt{2}}}="b2" 
		!{(2.25,0) }*+{\bullet_{\infty}}="c" 
		!{(3.25,0) }*+{\bullet_{-2}}="d" 
		"a1":^2"a2" 
		"b1":_2"b2" 
		"a2":^1"c"
		"b2":_1"c"
		"c":^1"d"
		"d":@(r,lu)_1 "d"
	}& $\displaystyle\frac{-2 z}{2 z^2-4 z+1}$\\ 
\hline

$\displaystyle\frac{6z^2 + 16z + 16}{-3z^2 - 4z - 4}$	 		& \xygraph{ 
		!{<0cm,0cm>;<2cm,0cm>:<0cm,1cm>::} 
		!{(0,0.5) }*+{\bullet_{0}}="a" 
		!{(1,0.5) }*+{\bullet_{-4}}="b" 
		!{(2,0.5) }*+{\bullet_{-4/3}}="c" 
		!{(0.5,-0.5) }*+{\bullet_{-2}}="d" 
		!{(1.5,-0.5) }*+{\bullet_{-1}}="e" 
		"a":^2"b" 
		"b":^1"c" 
		"c":@(r,lu)_1 "c"
		"d":@/_/_2"e" 
		"e":@/_/_1"d" 
	} &$\displaystyle\frac{3 z^2-4 z+1}{1-4 z}$ \\ 
\hline
\label{tab: trivial stab}
\end{tabular}
\end{table}
\end{center}

\begin{remark*}
This list of PCF maps raises some  questions.
\begin{enumerate}
\item
All maps except the sixth one and the last one  satisfy $\sigma_1 \in \left\{\pm 2, -6\right\}$.  (This is also true of the maps with nontrivial stabilizer described in Section~\ref{sec:symmlocus}.)  The line $\sigma_1 = 2$ in the moduli space of quadratic rational maps corresponds to the quadratic polynomials.  What (if anything) is special about these other two lines?
\item
Similarly, all maps except the sixth one and the last one correspond to integer values of $(\sigma_1, \sigma_2)$.  What is special about these the two anomalous maps?
\item
For the two anomalous maps we have $(\sigma_1, \sigma_2) = (-\frac 2 3, \frac 4 3)$ and  $(\sigma_1, \sigma_2) = (-\frac{10}3, \frac{20}3)$.  In other words, the symmetric functions of the multipliers have denominator at most~$3$ for all quadratic PCF maps defined over $\QQ$.  Is there some general phenomenon here that extends to maps defined over number fields?
\end{enumerate}
\end{remark*}

From~\cite[Proposition 4.73]{ads}, functions with trivial $\PGL_2$ stabilizer have no nontrivial twists.  That is, any quadratic PCF map defined over $\QQ$ with trivial stabilizer must be conjugate to one of the ten maps above, and the conjugacy must also be defined over $\QQ$.  Hence the rational preperiodic structures for these maps are invariant within the conjugacy class.    The possible structures, computed with Sage~\cite{sage}, appear in Table~\ref{tab: trivial preper}.

\begin{center}
\begin{table}
\caption{Preperiodic structures for quadratic maps with trivial stabilizer}

\begin{tabular}{|c|c|}\hline
$\phi(z)$		& Rational Preperiodic Points Graph				 \\ 
\hline \hline 

$z^2-2$			& \xygraph{ 
		!{<0cm,0cm>;<2cm,0cm>:<0cm,1cm>::} 
		!{(0,0) }*+{\bullet_{\infty}}="a" 
		!{(1,0) }*+{\bullet_{1}}="b" 
		!{(2,0) }*+{\bullet_{-1}}="c" 
		!{(0.5,-1) }*+{\bullet_{0}}="x" 
		!{(1.5,-1) }*+{\bullet_{-2}}="y" 
		!{(2.5,-1) }*+{\bullet_{2}}="z" 
		"a":@(l,ru) "a"
		"b":"c"
		"c":@(r,lu) "c"
		"x":"y"
		"y":"z"
		"z":@(r,lu) "z"
	}\\

\hline

$z^2-1$			& \xygraph{ 
		!{<0cm,0cm>;<2cm,0cm>:<0cm,1cm>::} 
		!{(0,0) }*+{\bullet_{\infty}}="a" 
		!{(1,0) }*+{\bullet_{1}}="b" 
		!{(2,0) }*+{\bullet_{0}}="c" 
		!{(3,0) }*+{\bullet_{-1}}="d" 
		"a":@(l,ru) "a"
		"b":"c"
		"d":@/_/"c"
		"c":@/_/"d"
	}\\

 \hline

$\displaystyle\frac{1}{2(z-1)^2}$	&\xygraph{ 
		!{<0cm,0cm>;<1.5cm,0cm>:<0cm,1cm>::} 
		!{(0,0) }*+{\bullet_{1}}="a" 
		!{(1,0) }*+{\bullet_{\infty}}="b" 
		!{(2,0) }*+{\bullet_{0}}="c" 
		!{(3,0) }*+{\bullet_{1/2}}="d" 
		!{(4,0) }*+{\bullet_{2}}="e" 
		!{(5,0) }*+{\bullet_{3/2}}="f"
		"a":"b" 
		"b":"c" 
		"c":"d"
		"f":"e"
		"d":@/_/"e" 
		"e":@/_/"d" 
	} \\
\hline

$\displaystyle\frac{1}{(z-1)^2}$	&  \xygraph{ 
		!{<0cm,0cm>;<2cm,0cm>:<0cm,1cm>::} 
		!{(0.7,0.5) }*+{\bullet_{\infty}}="a" 
		!{(0.7,-0.5) }*+{\bullet_{0}}="b" 
		!{(0,0) }*+{\bullet_{1}}="c" 
		!{(-.75,0) }*+{\bullet_{2}}="d"
		"a":@/^/"b"
    		"b":@/^/"c"
      		"c":@/^/"a"
		"d":"c"
	} \\
 \hline

$\displaystyle\frac{-1}{4z^2-4z}$	&  \xygraph{ 
		!{<0cm,0cm>;<2cm,0cm>:<0cm,1cm>::} 
		!{(0,0) }*+{\bullet_{1/2}}="a" 
		!{(1,0) }*+{\bullet_{1}}="b" 
		!{(2,0) }*+{\bullet_{\infty}}="c" 
		!{(3,0) }*+{\bullet_{0}}="d" 
		"a":"b" 
		"b":"c" 
		"c":@/_/"d" 
		"d":@/_/"c" 
	}  \\
\hline

$\displaystyle\frac{-4}{9z^2-12z}$	&\xygraph{ 
		!{<0cm,0cm>;<1.5cm,0cm>:<0cm,1cm>::} 
		!{(0,-0.25) }*+{\bullet_{2/3}}="a" 
		!{(1,-0.25) }*+{\bullet_{1}}="b" 
		!{(2,-0.25) }*+{\bullet_{4/3}}="c" 
		!{(3,-0.25) }*+{\bullet_{\infty}}="d" 
		!{(4,-0.25) }*+{\bullet_{0}}="e" 
		!{(2,0.75) }*+{\bullet_{1/3}}="f"
		!{(2,1.1) }*+{\ }="x"
		"a":"b" 
		"b":"c" 
		"c":"d"
		"f":"c"
		"d":@/_/"e" 
		"e":@/_/"d" 
	} \\
\hline

$\displaystyle\frac{2}{ (z-1)^2}$			& \xygraph{ 
		!{<0cm,0cm>;<2cm,0cm>:<0cm,1cm>::} 
		!{(0,0) }*+{\bullet_{1}}="a1" 
		!{(1,0) }*+{\bullet_{\infty}}="a" 
		!{(2,0) }*+{\bullet_{0}}="b" 
		!{(3,0) }*+{\bullet_{2}}="c" 
		"a1":"a"
		"a":"b" 
		"b":"c"
		"c" :@(r,lu) "c"
	}\\ 
\hline

$\displaystyle\frac{2z+1 }{ 4z - 2z^2}$	&\xygraph{ 
		!{<0cm,0cm>;<1.5cm,0cm>:<0cm,1cm>::} 
		!{(0,0) }*+{\bullet_{-1/2}}="c" 
		!{(1,0) }*+{\bullet_{0}}="d" 
		!{(2,0) }*+{\bullet_{\infty}}="e" 
		!{(3,0) }*+{\bullet_{2}}="f"
		"c":"d"
		"f":"e"
		"d":@/_/"e" 
		"e":@/_/"d" 
	} \\
\hline

$\displaystyle\frac{-2z}{ 2z^2 -4z+1}$			& \xygraph{ 
		!{<0cm,0cm>;<2cm,0cm>:<0cm,1cm>::} 
		!{(0,0) }*+{\bullet_{\infty}}="a" 
		!{(1,0) }*+{\bullet_{0}}="b" 
		"a":"b" 
		"b" :@(r,lu) "b"
	}\\ 
\hline

$\displaystyle\frac{3 z^2-4 z+1}{1-4 z}$	 		& \xygraph{ 
		!{<0cm,0cm>;<2cm,0cm>:<0cm,1cm>::} 
		!{(0,0.5) }*+{\bullet_{1/2}}="a" 
		!{(1,0.5) }*+{\bullet_{1/4}}="b" 
		!{(2,0.5) }*+{\bullet_{\infty}}="c" 
		!{(1,-0.5) }*+{\bullet_{-2}}="d" 
		!{(2,-0.5) }*+{\bullet_{-1}}="e" 
		!{(0,-0.5) }*+{\bullet_{1/3}}="f" 
		"a":"b" 
		"b":"c" 
		"c":@(r,lu) "c"
		"d":@/_/"e" 
		"a":"b" 
		"e":@/_/"d" 
		"f":"d" 
	} \\ 
\hline

\end{tabular}
\label{tab: trivial preper}

\end{table}
\end{center}

\section{PCF maps with nontrivial $\PGL_2$ stabilizer}\label{sec:symmlocus}
Quadratic rational maps with nontrivial $\PGL_2$ stabilizer  have been extensively studied.  In~\cite{milnrat}, Milnor described the symmetry locus for quadratic rational maps;  Manes investigated the arithmetic of these maps in~\cite{manespreper,manesmod}.  
  Jones and Manes found a height bound on PCF maps with nontrivial stabilizer and used that bound to show that over $\QQ$, the only maps meeting these criteria must be conjugate to either $\psi_1(z) = z^2$ or $\psi_2(z)= 1/z^2$~\cite[Proposition 5.1]{galois}.
 
 Unlike the six maps described in Section~\ref{sec:PCFnoSym}, these two maps have nontrivial twists.  That is, there are infinitely many $\PGL_2(\QQ)$ conjugacy classes within each of these two $\PGL_2(\bar \QQ)$ conjugacy classes of maps.  The different $\QQ$ conjugacy classes may have very different structures for their rational preperiodic points.  In this section, we find all of the possible rational preperiodic structures for these two conjugacy classes. 
 
  Throughout this section, $\zeta_n$ represents a primitive $n^{\text{th}}$ root of unity.

\begin{define}
If a point $\alpha \in \PP^1$ enters a cycle of least period $m$ after $n$ iterations  (i.e. if $\phi^n(\alpha)$ has period $m$ with $n$ and $m$ minimal), then $\alpha$ is called a periodic point of type $m_n$. 
\end{define}

\subsection{Maps conjugate to $\psi_1(z) = z^2$}
  Twists of $\psi_1$ are described completely in~\cite{manespreper}.  They  are given by
\[
\phi_b(z) = \frac{z}{2} + \frac{b}{z},
\]
where $b\neq 0$ is defined up to squares in $\QQ$.
Applying propositions  from~\cite{manespreper}, we  easily conclude:
\begin{enumerate}
\item
The map $\phi_b$ always has a rational fixed point at infinity and a rational point of type $1_1$ at $0$ (Propositions 1 and 5).  
\item
The map $\phi_b$ has finite rational  fixed points if and only if $b = \frac{c^2}{2}$ for $c\in\QQ^{\times}$ (Proposition 1), and all such maps are conjugate over $\QQ$.  Taking $b= \frac 12$ yields
\[
\phi_{1/2}(z) = \frac{z^2+1}{2z}.
\]
In this case, there are no additional points of type $1_1$ (Proposition 5).
\item
The map $\phi_b$ has rational points of primitive period 2 if and only if $b = -\frac {3c^2}2$ for $c\in\QQ^{\times}$ (Proposition 2), and all such maps are conjugate over $\QQ$.  So we take $b = -\frac 32$ to get 
\[
\phi_{-3/2}(z) = \frac{z^2-3}{2z}.
\]
In this case, we  have two rational points of type $2_1$ (Proposition 5) but no rational points of type $2_n$ for $n>1$ (Proposition 8) and no finite rational fixed points (Proposition 9).

\item
The map $\phi_b$ cannot have rational points of primitive period 3 or 4 (Theorems 3 and~4).  This will also follow Theorem~\ref{thm:z^2 periodic} below.

\item
The map $\phi_b$ has rational points of type $1_2$ if and only if $b = -\frac{c^2}2$ for $c\in\QQ^{\times}$, and all such maps are conjugate over $\QQ$.  So we take $b = -\frac 12$ to get the map
\[
\phi_{-1/2}(z) = \frac{z^2-1}{2z}.
\]
In this case, there are no finite rational fixed points (Proposition 6) and no rational points of period 2 (Proposition 9).

\item
The map $\phi_b$ cannot have rational points of type $1_n$ for $n\geq 3$ (Propositions 7 and 8).
\end{enumerate}

The description above yields four possible  rational preperiodic structures, shown in Table~\ref{table: twists 1}.

\begin{center}
\begin{table}
\caption{Preperiodic structures for twists of $\psi_1(z) = z^2$}

\begin{tabular}{|c|c|}
\hline
 $\displaystyle \phi_b(z) = \frac{z}{2} + \frac{b}{z}$		& Rational Preperiodic Points Graph \\ 
    \hline \hline 

$\displaystyle \phi_1(z) = \frac{z}{2} + \frac{1}{z}$				&
 \xygraph{
        !{<0cm,0cm>;<1cm,0cm>:<0cm,1cm>::}
        !{(1 ,0) }*+{\bullet_{\infty}}="a"
        !{(-1, 0) }*+{\bullet_{0}}="b"
        "a" :@(r,lu) "a"
        "b":"a"
	} \\
\hline

$\displaystyle \phi_{1/2}(z) = \frac{z}{2} + \frac{1}{2z}$				&
  \xygraph{
        !{<0cm,0cm>;<1cm,0cm>:<0cm,1cm>::}
         !{(-1 ,0) }*+{\bullet_{\infty}}="a"
        !{(-2, 0) }*+{\bullet_{0}}="b"
        !{(1,0) }*+{\bullet_{1}}="c"
         !{(2.5,0) }*+{\bullet_{-1}}="d"
       "a" :@(r,lu) "a"
        "b":"a"
        "c" :@(r,lu) "c"
        "d" :@(r,lu) "d"
    }\\
\hline

$\displaystyle \phi_{-3/2}(z) = \frac{z}{2} - \frac{3}{2z}$				&
\xygraph{
        !{<0cm,0cm>;<2cm,0cm>:<0cm,1cm>::}
        !{(-0.5 ,0) }*+{\bullet_{\infty}}="a"
        !{(-1.5, 0) }*+{\bullet_{0}}="b"
         !{(.5,-0.5) }*+{\bullet_{-2}}="c"
          !{(1,0.5) }*+{\bullet_{-1}}="d"
               !{(2.5,-0.5) }*+{\bullet_{2}}="e"
               !{(2,0.5)}*+{\bullet_{1}}="f"
        "a" :@(r,lu) "a"
        "b":"a"
         "c":"d"
               "e":"f"
       "d":@/_/"f"
       "f":@/_/"d"
    }\\ 
 \hline

$\displaystyle \phi_{-1/2}(z) = \frac{z}{2} - \frac{1}{2z}$				&
\xygraph{
        !{<0cm,0cm>;<1cm,0cm>:<0cm,1cm>::}
        !{(1 ,0) }*+{\bullet_{\infty}}="a"
        !{(-1, 0) }*+{\bullet_{0}}="b"
        !{(-2.3, .7) }*+{\bullet_{-1}}="c"
        !{(-2.3,-0.7)}*+{\bullet_{1}}="d"
        "a" :@(r,lu) "a"
        "b":"a"
        "d":"b"
        "c":"b"
    } \\
     \hline

\end{tabular}
\label{table: twists 1}
\end{table}
\end{center}

In order to claim we have a complete description of the possible rational preperiodic structures, we need the following result.

\begin{thm}\label{thm:z^2 periodic}
Let 
\[
\phi_b(z) = \frac{z}{2} + \frac{b}{z}.
\]
Then $\phi$ has no rational point of least period $n > 2$.
\end{thm}

\begin{proof}
Consider a point $\alpha \in \QQ$ so that $\alpha$ is periodic for $\phi_b(z)$.  Let 
\[
f(z) = \frac{z}{\sqrt{ 2b}}, \quad
\text{ so } \quad  \phi_b^f(z) = \phi_{1/2}(z) = \frac{z^2+1}{2z}.
\]
Then we have $f(\alpha) = \frac{\alpha}{\sqrt{2 b}}$ is periodic for $\phi_{1/2}(z)$.

Now let $g = (z-1)/(z+1)$.  It's a simple matter to check that $\psi_1(z):= \phi_{1/2}^g = z^2$, so that $g(f(\alpha)) \in \QQ\left[\sqrt {2b} \right]$ is periodic for $\psi_1(z)$.

We will now categorize periodic points for $\psi_1(z) = z^2$ that lie in quadratic fields, showing that none of them have period of length more than 2.  The result will follow. 

The map $\psi_1$ has a totally ramified fixed point at $\infty$.  Any finite periodic point of $\psi_1(z) = z^2$ is a root of $z^{2^n}-z$, so it is either $0$ or a root of $z^{2^n-1}-1$, i.e. a root of unity.  Since we seek periodic points that lie in quadratic fields, we can restrict our search to roots of unity that lie in quadratic fields, namely $\{ \pm 1, \pm i,  \zeta_3, \zeta_3^{-1},   \zeta_6, \zeta_6^{-1} \}$.

A computation verifies that  the preperiodic  structures for $\psi_1$containing these points  are the ones shown in Figure~\ref{fig:z^2pts}.  So  the only quadratic periodic points have period $1$ or $2$ as desired.
\end{proof}

\begin{figure}[h]
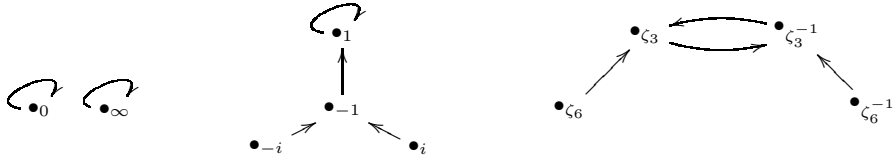

\centering
\mbox{
\subfigure{
    \xygraph{
        !{<0cm,0cm>;<2cm,0cm>:<0cm,1cm>::}
        !{(-1,1) }*+{\bullet_{0}}="j"
        !{(-.5,1) }*+{\bullet_{\infty}}="i"
        !{(1,2) }*+{\bullet_{1}}="a"
        !{(1,1) }*+{\bullet_{-1}}="b"
         !{(2.5,1) }*+{\bullet_{\zeta_6}}="c"
               !{(3,2) }*+{\bullet_{\zeta_3}}="d"
               !{(4.5,1) }*+{\bullet_{\zeta_6^{-1}}}="e"
               !{(4,2)}*+{\bullet_{\zeta_3^{-1}}}="f"
        !{(0.5,0.5) }*+{\bullet_{-i}}="g"
        !{(1.5,0.5) }*+{\bullet_{i}}="h"
        "a" :@(l,ru)  "a"
        "i" :@(l,ru)  "i"
        "j" :@(l,ru)  "j"
        "b":"a"
         "c":"d"
          "e":"f"
          "g":"b"
          "h":"b"
       "d":@/_/"f"
       "f":@/_/"d"
    }}
}
\caption{All possible quadratic periodic points for $\psi(z) = z^2$.}
\label{fig:z^2pts}
\end{figure}

\subsection{Maps conjugate to $\psi_2(z) = 1/z^2$}

From \cite{MMYY}, all such maps are conjugate over $\QQ$ to a map of the form 
\begin{equation}
\theta_{d,k}(z) = \frac{k z^2 - 2 d z +dk}{z^2 - 2 k z + d},
\qquad \text{with } k \in \QQ, \ d \in \QQ^{\times}, \text{ and } k^2\neq d.
\label{eqn:kdform}
\end{equation}
Conjugating this map by 
\[
f(z) = \frac{ z- \sqrt{d}}{z+ \sqrt{d}}
\quad \text{ yields } \quad
\theta_{d,k}^f(z) = \frac{t}{z^2} 
\quad\text{where } t = \frac{k - \sqrt d}{k + \sqrt d}.
\]
 
Conjugating this by $g(z) = t^{-1/3}z$ gives 
\[
\left(\theta_{d,k}^f\right)^g(z) = \frac{1}{z^2}.
\]

If $\alpha \in \QQ$ is preperiodic for $\theta_{d,k}$, then  $\beta=g^{-1}f^{-1}(\alpha)\in  \QQ(t^{1/3})$ is a preperiodic point for $\psi_2(z)$.  Since $[\QQ(\beta):\QQ] \leq 6$,  we may find all rational preperiodic structures for this family of maps by describing preperiodic points for $\psi_2$ of degree at most six.  Conjugating these points to lie in the rationals, we will find a map in the family with specified rational preperiodic points or show that none exists.

\begin{lem}\label{lem:psi2orbits}
All preperiodic points for $\psi_2(z) = 1/z^2$ of degree at most six are given in Figures~$\ref{fig:1/z^2ptsa}$--$\ref{fig:1/z^2ptsd}$.

\begin{figure}[h]
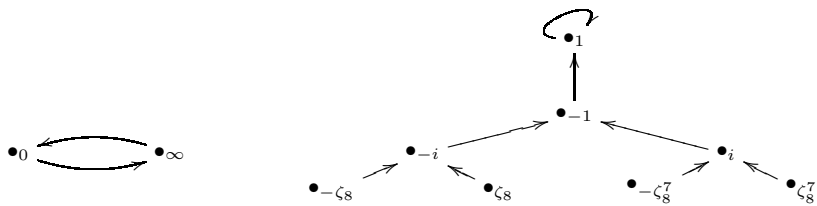

\centering
\mbox{
\subfigure{
    \xygraph{
        !{<0cm,0cm>;<2cm,0cm>:<0cm,1cm>::}
        !{(0,0) }*+{\bullet_{0}}="a"
        !{(1,0) }*+{\bullet_{\infty}}="b"
       "a":@/_/"b"
       "b":@/_/"a"
    }}
 \qquad\qquad
\subfigure{
    \xygraph{
        !{<0cm,0cm>;<2cm,0cm>:<0cm,1cm>::}
        !{(0,1.5) }*+{\bullet_{1}}="a"
        !{(0,.5) }*+{\bullet_{-1}}="b"
        !{(-1,0) }*+{\bullet_{-i}}="c"
        !{(1,0) }*+{\bullet_{i}}="d"
        !{(-1.6,-.5) }*+{\bullet_{-\zeta_8}}="e"
        !{(-.5,-.5) }*+{\bullet_{\zeta_8}}="f"
        !{(1.5,-.5) }*+{\bullet_{\zeta_8^7}}="g"
        !{(.5,-.5) }*+{\bullet_{-\zeta_8^7}}="h"
        "a" :@(l,ru)  "a"
         "b":"a"
         "c":"b"
         "d":"b"
         "e":"c"
         "f":"c"
         "g":"d"
         "h":"d"
    }}
}
\caption{$\psi_2(z) = 1/z^2$: Two cycle and one fixed point.}
\label{fig:1/z^2ptsa}
\end{figure}
\begin{figure}[h]
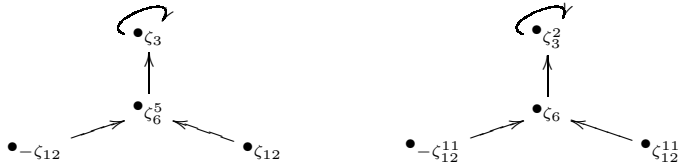

\centering
\mbox{
\subfigure{
    \xygraph{
        !{<0cm,0cm>;<2cm,0cm>:<0cm,1cm>::}
        !{(0,1.5) }*+{\bullet_{\zeta_3}}="a"
        !{(0,.5) }*+{\bullet_{\zeta_6^5}}="b"
        !{(-.75,0) }*+{\bullet_{-\zeta_{12}}}="c"
        !{(.75,0) }*+{\bullet_{\zeta_{12}}}="d"
        "a" :@(l,ru)  "a"
         "b":"a"
         "c":"b"
         "d":"b"
    }}
    \qquad\qquad
    \subfigure{
    \xygraph{
        !{<0cm,0cm>;<2cm,0cm>:<0cm,1cm>::}
        !{(0,1.5) }*+{\bullet_{\zeta_3^2}}="a"
        !{(0,.5) }*+{\bullet_{\zeta_6}}="b"
        !{(-.75,0) }*+{\bullet_{-\zeta_{12}^{11}}}="c"
        !{(.75,0) }*+{\bullet_{\zeta_{12}^{11}}}="d"
        "a" :@(l,ru)  "a"
         "b":"a"
         "c":"b"
         "d":"b"
    }}

}
\caption{$\psi_2(z) = 1/z^2$: Two additional fixed points.}
\label{fig:1/z^2ptsb}
\end{figure}
\begin{figure}[h]
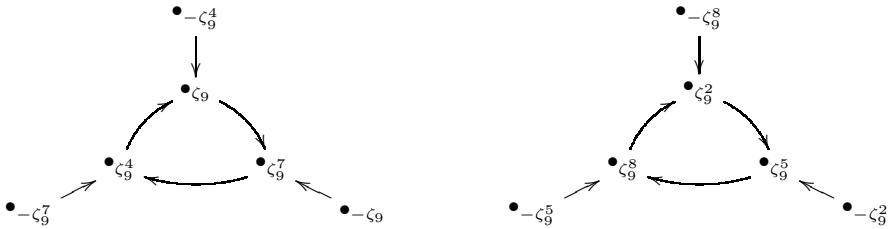

\centering
\mbox{
\subfigure{
    \xygraph{
        !{<0cm,0cm>;<2cm,0cm>:<0cm,1cm>::}
        !{(0,.5) }*+{\bullet_{\zeta_9}}="a"
        !{(.5,-.5) }*+{\bullet_{\zeta_9^7}}="b"
        !{(-.5,-.5) }*+{\bullet_{\zeta_{9}^{4}}}="c"
        !{(0,1.5) }*+{\bullet_{-\zeta_9^4}}="d"
        !{(1.1,-1.1) }*+{\bullet_{-\zeta_9}}="e"
        !{(-1.1,-1.1) }*+{\bullet_{-\zeta_9^7}}="f"
       "a":@/^/"b"
       "b":@/^/"c"
       "c":@/^/"a"
         "d":"a"
         "e":"b"
         "f":"c"
    }}
    \qquad\qquad
    \subfigure{
    \xygraph{
        !{<0cm,0cm>;<2cm,0cm>:<0cm,1cm>::}
        !{(0,.5) }*+{\bullet_{\zeta_9^2}}="a"
        !{(.5,-.5) }*+{\bullet_{\zeta_9^5}}="b"
        !{(-.5,-.5) }*+{\bullet_{\zeta_{9}^{8}}}="c"
        !{(0,1.5) }*+{\bullet_{-\zeta_9^8}}="d"
        !{(1.1,-1.1) }*+{\bullet_{-\zeta_9^2}}="e"
        !{(-1.1,-1.1) }*+{\bullet_{-\zeta_9^5}}="f"
       "a":@/^/"b"
       "b":@/^/"c"
       "c":@/^/"a"
         "d":"a"
         "e":"b"
         "f":"c"
    }}
}
\caption{$\psi_2(z) = 1/z^2$: Two three-cycles.}
\label{fig:1/z^2ptsc}
\end{figure}
\begin{figure}[h]
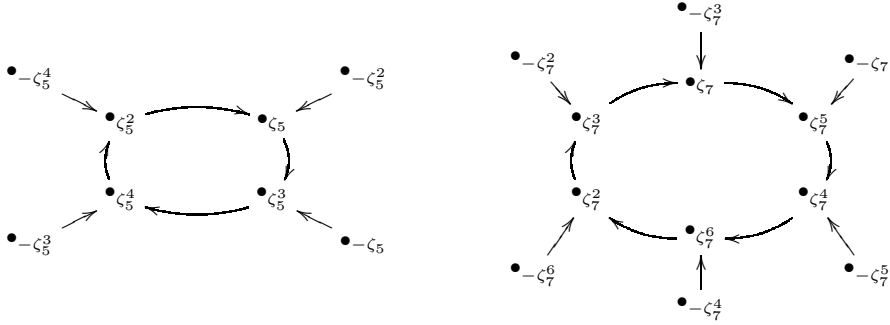

\centering
\mbox{
\subfigure{
    \xygraph{
        !{<0cm,0cm>;<2cm,0cm>:<0cm,1cm>::}
        !{(.5,.5) }*+{\bullet_{\zeta_5}}="a"
        !{(.5,-.5) }*+{\bullet_{\zeta_5^3}}="b"
        !{(-.5,-.5) }*+{\bullet_{\zeta_{5}^{4}}}="c"
        !{(-.5,.5) }*+{\bullet_{\zeta_{5}^{2}}}="d"
        !{(1.1,1.1) }*+{\bullet_{-\zeta_5^2}}="e"
        !{(-1.1,1.1) }*+{\bullet_{-\zeta_{5}^4}}="f"
        !{(1.1,-1.1) }*+{\bullet_{-\zeta_5}}="g"
        !{(-1.1,-1.1) }*+{\bullet_{-\zeta_5^3}}="h"
       "a":@/^/"b"
       "b":@/^/"c"
       "c":@/^/"d"
       "d":@/^/"a"
         "e":"a"
         "f":"d"
         "g":"b"
         "h":"c"
    }}
    \qquad\qquad
    \subfigure{
    \xygraph{
        !{<0cm,0cm>;<2cm,0cm>:<0cm,1cm>::}
        !{(0,1) }*+{\bullet_{\zeta_7}}="a"
        !{(.75,.5) }*+{\bullet_{\zeta_7^5}}="b"
        !{(.75,-.5) }*+{\bullet_{\zeta_7^4}}="c"
        !{(-.75,-.5) }*+{\bullet_{\zeta_{7}^{2}}}="d"
        !{(-.75,.5) }*+{\bullet_{\zeta_{7}^{3}}}="e"
        !{(0,-1) }*+{\bullet_{\zeta_{7}^{6}}}="f"
        !{(0,1.95) }*+{\bullet_{-\zeta_7^3}}="g"
        !{(1.1,1.3) }*+{\bullet_{-\zeta_7}}="h"
        !{(-1.1,1.3) }*+{\bullet_{-\zeta_{7}^{2}}}="i"
        !{(1.1,-1.5) }*+{\bullet_{-\zeta_7^5}}="j"
        !{(-1.1,-1.5) }*+{\bullet_{-\zeta_7^6}}="k"
        !{(0,-1.95) }*+{\bullet_{-\zeta_7^4}}="l"
       "b":@/^/"c"
       "c":@/^/"f"
       "f":@/^/"d"
       "d":@/^/"e"
       "e":@/^/"a"
       "a":@/^/"b"
         "g":"a"
         "h":"b"
         "i":"e"
         "j":"c"
         "k":"d"
         "l":"f"
    }}
}
\caption{$\psi_2(z) = 1/z^2$: A four-cycle and a six-cycle.}
\label{fig:1/z^2ptsd}
\end{figure}

\end{lem}

\begin{proof}
For $n$ even, $\psi_2^n(z) = z^{2^n}$, so the points with period dividing $n$ are $0$, $\infty$, and  $(2^n-1)^{\textup{st}}$ roots of unity.  For $n$ odd,  $\psi_2^n(z) =z^{-2^n}$, so the points with period dividing $n$ are $(2^n+1)^{\textup{st}}$ roots of unity.  Hence all strictly periodic points other than $0$ and $\infty$ are roots of unity of odd order.
The only roots of unity of odd order with degree no more than $6$ are powers of $\{1, \zeta_3, \zeta_5, \zeta_7, \zeta_9\}$.  We may find their periodic structures by  iterating $\psi_2$ with the appropriate seed values.

Let $\beta$ be a preperiodic point for $\psi_2$.  Then $[\QQ(\beta):\QQ] \leq 6$ if and only if all powers of $\beta$ also satisfy $[\QQ(\beta^n):\QQ] \leq 6$.  In particular, the orbit of $\beta$ lands in some cycle, and the points of that cycle have degree no more than six.  Hence we can find all preperiodic points for $\psi_2$ having degree no more than six by finding preimages of the periodic points described above, and continuing until the field generated by the preimages has degree greater than six.
It is a simple matter to verify that this process yields the diagrams given.
\end{proof}

\begin{prop}\label{prop:2cyc}
Let  $\phi(z) \in \QQ$ be conjugate over $\overline \QQ$ to $\psi_2$.  Then $\phi$ has no points of type $2_n$ for $n\geq1$.  For $m\neq 2$, $\phi$ has the same number of rational points of type $m_1$ as it has rational points of primitive period $m$.
\end{prop}

\begin{proof}
The critical points of $\psi_2$  lie on a two cycle, and this property is preserved under conjugation.  Therefore each critical point is also a critical value, so if the critical points of $\phi$ are $\{ \gamma_1, \gamma_2 \}$ we have $\phi^{-1}(\gamma_i) = \{ \gamma_j \}$ for $i\neq j$.  Hence $\phi$ has no points of type $2_1$ and it follows that $\phi$ has no points of type $2_n$ for $n\geq 1$.

Let $\alpha$ be a rational point of primitive period $m$ for $\phi$.  Then all points on the $m$-cycle containing $\alpha$ are also rational since $\phi(z) \in \QQ(z)$.  Therefore the quadratic $\phi(z) = \alpha$ has one rational root.  Since $m \neq 2$, $\alpha$ is not one of the critical values of $\phi$ by the argument above.  Hence, the quadratic $\phi(z) = \alpha$ has two distinct roots and both must be rational.  That is, there is a rational point $\beta$ not on the $m$-cycle satisfying $\phi(\beta) = \alpha$, and $\beta$ is a point of type $m_1$.
\end{proof}

\begin{prop}
Let $\phi(z) \in \QQ(z)$ be conjugate to $\psi_2$.  If $\phi$ has a rational two-cycle then it may have either no rational fixed points or one rational fixed point.  In either case, it has no other rational preperiodic points except the required point of type $1_1$.
\end{prop}

\begin{proof}
From \cite[Lemma 5.1]{MMYY},  we see that $\phi$ has a rational two-cycle if and only if it is conjugate over $\QQ$  to $\theta_t(z) = \frac{t}{z^2}$ for some $t \in \QQ^{\times}$.  Solving $\theta_t(z) = z$, we see that there is a rational fixed point if and only if $t \in ( \QQ^{\times})^3$, and all such maps are conjugate over $\QQ$.  

Furthermore, if $f(z) = t^{-1/3}z$, then $\theta_t^f(z) = \psi_2$.  Applying $f$ to the preperiodic structures given in Lemma~\ref{lem:psi2orbits}, we find no other rational preperiodic points.
\end{proof}

By Proposition~\ref{prop:2cyc}, we have only two rational preperiodic structures for maps conjugate to $\phi_2(z)$ that contain rational points of primitive period 2.  These are the first two maps represented in Table~\ref{table: twists 2}.

\begin{prop}\label{prop:fpstruct}
Let  $\phi(z) \in \QQ$ be conjugate over $\overline \QQ$ to $\psi_2$.  Suppose $\phi$ has no rational points of period $n>1$.  Then $\phi$ has one of the following rational preperiodic structures:
\begin{enumerate}[\textup(a\textup)]
\item
$\phi$ has no rational fixed points \textup(hence no rational preperiodic points at all\textup);
\item
$\phi$ has exactly one rational fixed point and one point of type $1_1$ but no other rational preperiodic points;
\item
$\phi$ has exactly one rational fixed point, one rational point of type $1_1$, and two rational points of type $1_2$, with no other rational preperiodic points; or
\item
$\phi$ has exactly three rational fixed points and three rational points of type $1_1$, with no other rational preperiodic points.
\end{enumerate}
\end{prop}

\begin{proof}
Choosing $k=1$ and $d=2$ in the normal form from equation~\eqref{eqn:kdform} yields the map
\[
\frac{z^2-4 z+2}{z^2-2 z+2}.
\]
One can check computationally that this map has no rational points of primitive period~1, 2, 3, 4, or~6.  By Lemma~\ref{lem:psi2orbits}, these are the only possibilities.

Choosing $k=0$ and $d=2$ in the normal form from equation~\eqref{eqn:kdform} yields the map
\[
-\frac{4 z}{z^2+2}.
\]
One can check computationally that this map has fixed point~$0$ and no other rational points of primitive period~1, 2, 3, 4, or~6.  By Lemma~\ref{lem:psi2orbits}, these are the only possibilities.  We also have $\infty \mapsto 0$, a rational point of type $1_1$.  The preimages of $\infty$ are not rational, so there are no other rational preperiodic points.

Beginning with the preperiodic structure described in Lemma~\ref{lem:psi2orbits}, we see that conjugating $\phi_2$ by any $f \in \PGL_2$ which maps three arbitrary rational points to $1$, $i$, and $-i$ creates a map with rational type $1_2$ points.  Choose 
\[
f(z) = \frac{i z+1}{z+i} \quad \text{which yields } \psi_2^f(z) = \frac{-z^2+2 z+1}{z^2+2 z-1}.
\]
One can check computationally that this map has no rational point of period 2, 3, 4, or~6.  The only rational fixed point is 1;  $-1$ is a type $1_1$ point;  and $0$ and $\infty$ are type $1_2$ points.  There are no rational type $1_3$ points.

Again, beginning with the preperiodic structure described in Lemma~\ref{lem:psi2orbits}, we see that conjugating $\phi_2$ by any $f \in \PGL_2$ which maps three arbitrary rational points to $1$, $\zeta_3$, and $\zeta_3^2$.  yields a map with three rational fixed points.  Choose 
\[
f(z) =\frac{(1 + \zeta_3) z + 1}{z - \zeta_3^2}
\quad \text{which yields } \psi_2^f(z) = -\frac{(z-2) z}{2 z-1}.
\]
This map has fixed points at 0, 1, and $\infty$ and the corresponding rational type $1_1$ points.
One can check computationally that this map has no rational point of period 2, 3, 4, or~6, and no rational type $1_2$ points.

We have shown that each of the possibilities listed are possible for maps conjugate to $\psi_2$.  It remains to check that no other possibilities exist.  

Since $\phi$ is defined over $\QQ$, the cubic polynomial $\phi(z) = z$ has either zero, one, or three rational roots.  Hence we cannot have exactly two rational fixed points.

If a map $\phi$ is conjugate to $\psi_2$ and has rational points of type $1_3$, then it is conjugate over $\QQ$ to a map with $\QQ$ to a map with a fixed point at $1$ and the type $1_2$ points at $0$ and $\infty$.  We found such a map above, and it does not have  rational type $1_3$ points. 

Similarly, if a map $\phi$ is conjugate to $\psi_2$ and has three rational fixed points and rational points of type $1_2$, then it is conjugate over $\QQ$ to a map with one fixed point at $1$ and its type $1_2$ points at $0$ and $\infty$.  We found such a map above, and it does not have additional rational fixed points. 
We have now exhausted all possibilities.
\end{proof}

By Proposition~\ref{prop:fpstruct}, the  third through sixth rational preperiodic structures in  Table~\ref{table: twists 2}
 are the only ones possible for maps conjugate to $\psi_2$ that have no rational points of least period $n>1$.

\begin{prop}\label{prop:3cyc}
Let  $\phi(z) \in \QQ$ be conjugate over $\overline \QQ$ to $\psi_2$.  Suppose $\phi$ has a rational point of period $3$.  Then $\phi$ has exactly three such points and three points of type $3_1$.  The map  $\phi$ has no other rational preperiodic points.
\end{prop}

\begin{proof}
If $\phi$ is conjugate to $\psi_2$ and has a rational point of period $3$, then it is conjugate over $\QQ$ to a map with the three cycle $0 \mapsto 1 \mapsto \infty \mapsto 0 \mapsto \cdots$.  This conjugacy completely specifies the map.  Given the preperiodic structure described in Lemma~\ref{lem:psi2orbits}, we may begin with $f \in \PGL_2$ which maps $0$, $1$, and $\infty$ to $\zeta_9$, $\zeta_9^7$, and $\zeta_9^6$.  This is 
\[
f = \frac{\zeta_9^6 z - \zeta_9^7}{-z + \zeta_9^4}
\quad \text{ which yields } \psi_2^f(z) = \frac{2 z-1}{z^2-1}.
\]
One may verify computationally that this map has the desired three cycle and no other rational points of period 1, 2, 3, 4, or~6.  It has rational type $3_1$ points mapping into the three cycle, but the type $3_2$ points are not rational.
\end{proof}

By Proposition~\ref{prop:3cyc}, there is only one rational preperiodic structure for maps conjugate to $\psi_2$ that have a rational point of primitive period $3$.  This is the last map in Table~\ref{table: twists 2}.

\begin{prop}\label{prop: cyc}
Let  $\phi(z) \in \QQ$ be conjugate over $\overline \QQ$ to $\psi_2$.  Then $\phi$ has no rational points of period $n > 3$.
\end{prop}

\begin{proof}
If $\phi$ has  rational points of period~$4$, then it is conjugate over $\QQ$ to a map where three of those points are at $0$, $1$, and $\infty$.  Applying Lemma~\ref{lem:psi2orbits}, we choose $f \in \PGL_2$ mapping these three rational points to three powers of $\zeta_5$.  Conjugating $\psi_2$ by this map does not yield a map defined over $\QQ$.  
The argument for points of period~6 is the same, but using powers of $\zeta_7$.
\end{proof}

\begin{center}
\begin{table}
\caption{Preperiodic structures for twists of $\psi_2(z) = 1/z^2$}

\begin{tabular}{|c|c|}
\hline
$\phi(z)$		& Rational Preperiodic Points Graph				 \\ 
\hline \hline 

$\displaystyle \frac{1}{z^2}$			&
 \xygraph{
        !{<0cm,0cm>;<2cm,0cm>:<0cm,1cm>::}
        !{(0,0) }*+{\bullet_{1}}="a"
        !{(-1,0) }*+{\bullet_{-1}}="b"
         !{(1,0) }*+{\bullet_{0}}="c"
         !{(2,0)}*+{\bullet_{\infty}}="d"
        "a" :@(r,lu)  "a"
        "b":"a"
       "c":@/_/"d"
       "d":@/_/"c"
    } \\ 
 \hline

$\displaystyle \frac{2}{z^2}$			&
\xygraph{
        !{<0cm,0cm>;<2cm,0cm>:<0cm,1cm>::}
               !{(1,0) }*+{\bullet_{0}}="c"
               !{(2,0)}*+{\bullet_{\infty}}="d"
       "c":@/_/"d"
       "d":@/_/"c"
    }\\
\hline

$\displaystyle\frac{z^2-4 z+2}{z^2-2 z+2}$			&
no rational preperiodic points
    \\ 
\hline

$\displaystyle -\frac{4 z}{z^2+2}$	&\xygraph{
        !{<0cm,0cm>;<2cm,0cm>:<0cm,1cm>::}
        !{(0,0) }*+{\bullet_{0}}="a"
        !{(-1,0) }*+{\bullet_{\infty}}="b"
        "a" :@(r,lu)  "a"
        "b":"a"
            } \\
     \hline

$\displaystyle \frac{-z^2+2 z+1}{z^2+2 z-1}$	&  
\xygraph{
        !{<0cm,0cm>;<1cm,0cm>:<0cm,1cm>::}
        !{(1 ,0) }*+{\bullet_{1}}="a"
        !{(-1, 0) }*+{\bullet_{-1}}="b"
        !{(-2.3, .7) }*+{\bullet_{0}}="c"
        !{(-2.3,-0.7)}*+{\bullet_{\infty}}="d"
        "a" :@(r,lu) "a"
        "b":"a"
        "d":"b"
        "c":"b"
    }  \\
   \hline

$\displaystyle -\frac{(z-2) z}{2 z-1}$ &  \xygraph{
        !{<0cm,0cm>;<2cm,0cm>:<0cm,1cm>::}
        !{(-2,0) }*+{\bullet_{0}}="a"
        !{(-3,0) }*+{\bullet_{2}}="b"
        !{(-0.25,0) }*+{\bullet_{1}}="c"
        !{(-1.25,0) }*+{\bullet_{-1}}="d"
        !{(1.5,0) }*+{\bullet_{\infty}}="e"
        !{(.5,0) }*+{\bullet_{\frac 1 2}}="f"
        "a" :@(r,lu)  "a"
        "c" :@(r,lu)  "c"
        "e" :@(r,lu)  "e"
         "b":"a"
         "d":"c"
         "f":"e"
    } \\
 \hline

$\displaystyle \frac{2 z-1}{z^2-1}$	& \xygraph{
        !{<0cm,0cm>;<2cm,0cm>:<0cm,1cm>::}
        !{(0,.5) }*+{\bullet_{0}}="a"
        !{(.5,-.5) }*+{\bullet_{1}}="b"
        !{(-.5,-.5) }*+{\bullet_{\infty}}="c"
        !{(0,1.5) }*+{\bullet_{\frac 1 2}}="d"
        !{(1.1,-1.1) }*+{\bullet_{2}}="e"
        !{(-1.1,-1.1) }*+{\bullet_{-1}}="f"
         !{(0,1.75) }*+{\ }="x"
       "a":@/^/"b"
       "b":@/^/"c"
       "c":@/^/"a"
         "d":"a"
         "e":"b"
         "f":"c"
    }\\
  \hline

\end{tabular}
\label{table: twists 2}
\end{table}
\end{center}

\begin{acknowledgements}
The authors thank  Xander Faber,  Patrick Ingram, Rafe Jones,  and Kevin Pilgrim for helpful comments.  The second author is grateful to the Institute for Computational and Experimental Research in Mathematics (ICERM) for a very productive semester program in complex and arithmetic dynamics during which much of this work was completed, and especially to Ben Hutz and the rest of the Sage projective space working group for the useful code on which our algorithm was originally built.

The algorithm in Section~\ref{sec:algorithm} and complete list of PCF maps constitutes a portion of the third author's Ph.D.~thesis.
\end{acknowledgements}

\newpage

\bibliographystyle{plain}

\affiliationone{
   David Lukas \\
   University of \Hawaii\ at \Manoa \\
   USA
   \email{david.lukas@hawaii.edu\\
   }}
\affiliationtwo{
   Michelle Manes and Diane Yap\\
   Department of Mathematics,\\
   University of \Hawaii\ at \Manoa \\
   USA
   \email{mmanes@math.hawaii.edu\\
      dianey@gmail.com
   }}

\end{document}